\documentclass{amsart}
\usepackage{verbatim,amssymb,amsmath,amscd,latexsym,amsbsy,mathrsfs, color} 
\usepackage{graphicx}
\input{xy}
\xyoption{all}

\newtheorem{thm}{Theorem}[section]
\newtheorem{cor}[thm]{Corollary}

\newtheorem{rmk}[thm]{Remark}

\DeclareMathOperator*{\im}{im} 
\DeclareMathOperator*{\coker}{coker} 

\DeclareMathOperator*{\rank}{rank}
\DeclareMathOperator*{\Spec}{Spec}

\DeclareMathOperator*{\Hodge}{Hodge}
\DeclareMathOperator*{\Tate}{Tate}
\DeclareMathOperator*{\supp}{supp}

\DeclareMathOperator*{\tr}{tr}
\DeclareMathOperator*{\Cor}{Cor}
\DeclareMathOperator*{\MHS}{MHS}

\newcommand{\PP}{\mathbb{P}}

\newcommand{\cA}{\mathcal{A}}

\newcommand{\cL}{\mathcal{L}}

\newcommand{\cZ}{\mathcal{Z}}

\newcommand{\lOmega}{\underline{\Omega}}

\newcommand {\C} {{\mathbb C}}
\newcommand {\R} {{\mathbb R}}
\newcommand {\Z} {{\mathbb Z}}
\newcommand {\Q} {{\mathbb Q}}
\newcommand {\HH} {{\mathbb H}}

\newcommand {\E} {{\mathcal E}}
\newcommand {\dt} {{\bullet}}

\newcommand {\OO} {{\mathcal O}}

 \newtheorem{lemma}{Lemma}[section]
 \newtheorem{prop}{Proposition}[section]

\numberwithin{equation}{section}

\begin{document}
\title{Hodge cycles and the Leray filtration}
\author{
        Donu Arapura    
}
\thanks{Partially supported by a grant from the Simons foundation}
\address{Department of Mathematics\\
Purdue University\\
West Lafayette, IN 47907\\
U.S.A.}
\maketitle 

\begin{abstract}
 Given a fibered variety, we
  pull back the Leray filtration to the Chow group, and use this  to
  give some criteria for the Hodge and Tate conjectures to hold for
  such varieties. We show that the Hodge
  conjecture holds for a good desingularization of a self fibre
  product of a nonisotrivial elliptic surface under appropriate conditions. We also show that the
  Hodge and Tate conjectures hold for natural families of abelian
  varieties parameterized by certain Shimura curves.
   \end{abstract}

In \cite{arapura2}, we checked the Hodge and Tate conjectures for the universal family of genus two curves with level structure
by showing the absence of  these cycles in a critical piece of the
Leray spectral sequence. In this article, we will to take this idea further.
Given a surjective map of smooth projective varieties $f:X\to Y$, let  $U\subseteq Y$ be the complement of the discriminant and $V=f^{-1}U$.
Then we define a filtration on the rationalized Chow group such that we have a cycle map
$$\Box^{p,i}:Gr^i_L CH^p(V) \to H^{i}(U, R^{2p-i} f_*\Q)$$
By  \cite{arapura}, the right side carries a mixed Hodge structure.
We have that $\Box^{p,i}$ lands within the space of Hodge cycles. Let
$m=\dim Y$ and $r=\dim X-\dim Y$.
Our first theorem is that the Hodge conjecture, in Jannsen's sense, holds for $V$,
if  $\Box^{p,i}$ surjects onto the space of Hodge cycles for all $(p,i)$ in the parallelogram bounded by
the lines $i=0$, $i=m$, $i=2p$, and $i=2p-r$.  An analogous statement holds for the Tate conjecture. As a corollary, we show that when $X$ is a fourfold,
the  Hodge conjecture holds for $X$ if $\Box^{2,m}$ surjects onto the space of Hodge cycles. A special case of this was used in \cite{arapura2}.

In the rest of this article, we turn our attention to  particular classes of examples.
Suppose that $f_\lambda:\E_\lambda\to Y_\lambda$ is a family of nonisotrivial   elliptic surfaces,
such that the ``Kodaira-Spencer" map 
$$\kappa_\lambda:Gr^1_F H^1(U_\lambda, R^1f_{\lambda*}\C)\to Gr^0_F H^1(U_\lambda, R^1f_{\lambda*}\C)\otimes \Omega^1_{\lambda}$$
is injective somewhere. Theorem \ref{thm:main} shows
that the Hodge conjecture holds for the complement  $V_\lambda\subset \E_\lambda\times_{Y_\lambda} \ldots \times_{Y_\lambda}\E_\lambda$
of the set of singular fibres, when $\lambda$ is very general. 
For a zero dimensional family,  injectivity of $\kappa_\lambda$ holds
if and only if the elliptic surface is extremal in the sense
that the Picard number is maximal and the Mordell-Weil rank is zero (lemma \ref{lemma:extremal}).
Elliptic modular surfaces are extremal, so we recover  a result  of Gordon \cite{gordon} that theorem \ref{thm:main} 
holds for such surfaces. Our  theorem also  holds for certain
families of nonextremal elliptic K3 surfaces constructed by Dolgachev
\cite{dolgachev} and Hoyt \cite{hoyt}. Gordon's argument for elliptic
modular surfaces does not generalize. Our proof  is different, and is based on analyzing
 $\Box^{p,i}$ for $i=0,1$.  The  map $\Box^{p,0} $ is shown to surject onto the space of Hodge
 cycles by the K\"unneth formula and invariant theory. On the other hand  for very general $\lambda$, we will show $\Box^{p,1}$ is trivially surjective in the sense that there are no Hodge cycles on the target. This will be deduced by showing that   injectivity  of $\kappa_\lambda$
implies injectivity of the  Kodaira-Spencer map
$$Gr^p_F H^{1}(U_\lambda, R^{2p-1} F_{\lambda*}\C)\to Gr^{p-1}_F H^{1}(U_\lambda, R^{2p-1} F_{\lambda*}\C)$$
 where $F_\lambda:V_\lambda\to Y_\lambda$ is the structure map.

Suppose that $Y$ is a Shimura curve associated to a quaternion
division algebra
split at exactly one real place over a totally real field. The curve
$Y$ can be viewed as a moduli space of abelian varieties with some extra structure.
In particular, for sufficiently high level, $Y$ carries a natural family of abelian varietes $f:\cA\to Y$. Theorem \ref{thm:SCHodge} shows that the Hodge 
and Tate conjectures
holds for a fibre product $X=\cA\times_Y\ldots \times_Y \cA\xrightarrow{F} Y$. The proof again involves checking surjectivity
of $\Box^{p,i}$ onto the space of Hodge cycles. When $i=0$ or $2$, this follows from invariant theory. When $i=1$,
work of Viehweg and Zuo \cite{vz}  is used to show that in fact
$$Gr^p_FH^1(Y, R^{2p-1}F_*\C)=0$$
Therefore $\Box^{p,1}$ is again trivially surjective.

My thanks to the referee catching a serious error in the  first version of this paper.

\section{Hodge cycles and the Leray filtration}

Let $\MHS$ denote the category of rational mixed Hodge structures, and
$\MHS^p$ the full subcategory of polarizable Hodge structures.
We note that mixed Hodge structures of geometric origin in fact lie in
$\MHS^p$, and the category of pure polarizable Hodge structures is semisimple
\cite{beilinson}.
Given a mixed Hodge structure $H$, set
$$\Hodge(H) := Hom_{\MHS}(\Q(0), H)$$
The following facts will often be used without comment below.

\begin{lemma}\label{lemma:HodgeFun}
\-
\begin{enumerate}
\item There is an injection
  $$\Hodge(H) \hookrightarrow \Hodge( Gr^W_0H)$$
  It is an isomorphism if $H$ has nonnegative weight, i.e. $W_{-1}H=0$.

\item $\Hodge(-)$ is an exact functor on the category of polarizable
  mixed Hodge structures of nonnegative  weight.

  \item If
  $$H_1\to H_2\to H_3$$
  is an exact sequence of mixed Hodge structures, such
  that $H_1$ is polarizable with nonnegative weight, then
  $$\Hodge(H_1)\to\Hodge( H_2)\to \Hodge(H_3)$$
  is exact.
  
  \item If $\phi:H_1\to H_2$ is a morphism of polarizable Hodge
    structures, such that $Gr^0_FH_1\to Gr_F^0 H_2$ is surjective
    (resp. injective)
    then $\Hodge(H_1)\to \Hodge(H_2)$ is surjective (resp. injective).
\end{enumerate}
\end{lemma}

\begin{proof}
 The first item follows from  strictness of the weight filtration  \cite[thm 2.3.5]{deligneH}.
 The second follows from (1) and the semisimplicity of the category of
 pure polarizable Hodge structures. Statement (3) follows from (1) and
 (2), and the observation that subobjects and quotients of $H_1$ are
 also polarizable with nonnegative weight. 
 Given $\phi:H_1\to H_2$ inducing
 a surjective map $Gr^0_FH_1\to Gr_F^0 H_2$, let $H_3=\coker\phi$. Then
 $$\coker(\Hodge(H_1)\to \Hodge(H_2) )=\Hodge(H_3)\subseteq Gr_F^0H_3=0$$ 
 The  injectivity part of (4) is similar. 
\end{proof}

If $X$ is an algebraic variety, let $CH_p(X)$ (resp. $CH^p(X)$)
denote the Chow group of $p$ dimensional (resp. codimension $p$)
cycles tensored with $\Q$. When $X$ is defined over $\C$,
 Borel-Moore homology $H_i(X,\Q)$ carries a mixed Hodge structure dual
 to the one on
compactly supported cohomology $H_c^i(X,\Q)$.
We have a cycle map 
$$\Box_p :CH_p(X)\to H_{2p}(X,\Q)$$
such that the image lands in
$\Hodge(H_{2p}(X,\Q)(-p))$, c.f. \cite[chap 19]{fulton}, \cite[\S 5, \S7]{jannsen}. We will usually write $[Z]$ instead of $\Box_p(Z)$.
Following Jannsen \cite{jannsen}, we say that the Hodge conjecture holds
for $X$ if the cycle map 
$$CH_p(X)\to \Hodge(H_{2p}(X,\Q)(-p))$$
is surjective  for
all $p$. When $X$ is smooth,  we can apply Poincar\'e duality to  identify this with the cycle map
$$\Box^p:CH^p(X)\to \Hodge(H^{2p}(X,\Q)(p)) $$
We  recall that $\Box^*$ is a graded ring homomorphism.

Let us now suppose that $K$ is a 
finitely generated field of characteristic $0$, and let $G_K = Gal(\bar K/K)$.
Suppose that
$X$ is defined over   $K$, and let $\bar X= X\times_{\Spec K}\bar K$.
Given a $G_K$-module $H$, define the space of Tate cycles
in $H$ by
$$\Tate(H)  =\sum_{L/K \text{ finite}}
H^{G_L} $$
Fix a prime $\ell$.
Again following Jannsen, we say that Tate's conjecture holds for $X$ if
the cycle map
$$\Box_{p,\ell}: CH_p(\bar X)\otimes \Q_\ell\to \Tate(H_{2p}(\bar X_{et},
\Q_\ell)(-p))$$
is surjective for each $p$.
When $X$ is smooth, this is equivalent to the more familiar statement
that the cycle map
$$\Box^{p}_{\ell}:CH^p(\bar X)\otimes \Q_\ell\to \Tate(H^{2p}(\bar X_{et},
\Q_\ell)(p))$$
is surjective for each $p$. We note that $\Box^p$ and $\Box^p_\ell$ are compatible
under the Artin comparison isomorphism.

We note that Jannsen's form of the Hodge conjecture follows from the
usual Hodge conjecture, so it is not stronger. More precisely, we have:

\begin{thm}[Jannsen]\label{thm:jannsen}
  If  the Hodge conjecture holds for a desingularization of a
compactification of $X$, then it holds for $X$ 
\end{thm}
\begin{proof}
This follows from the proof of \cite[thm 7.9]{jannsen}.
\end{proof}

\begin{lemma}\label{lemma:localization}
 Let $U\subset X$ be Zariski open, and $Z=X-U$. If $X$ is projective and the Hodge conjecture holds for $U$ and $Z$, then it holds for $X$.
 If the Hodge conjecture holds for $X$, then it holds for $U$.
 \end{lemma}

 \begin{proof}
   We have a localization sequence
   $$H_{2p}(Z)\to H_{2p}(X)\to H_{2p}(U)$$
   where $H_*$ denotes Borel-Moore homology.
  Projectivity of $X$ implies that, after twisting by $\Q(-p)$, the weight of the first term
  becomes nonnegative. So we deduce
   from lemma~\ref{lemma:HodgeFun} that
   $$ \Hodge(H_{2p}(Z)(-p))\to \Hodge(H_{2p}(X)(-p))\to
   \Hodge(H_{2p}(U)(-p))$$
   is exact. 
   By \cite[\S 1.8, chap 19]{fulton}, this fits into a larger
   commutative diagram with exact rows
 $$
\xymatrix{
 CH_p(Z)\ar[r]\ar[d]^{\alpha} & CH_p(X)\ar[r]\ar[d] & CH_p(U)\ar[r]\ar[d]^{\beta} & 0 \\ 
 \Hodge(H_{2p}(Z)(-p))\ar[r] & \Hodge(H_{2p}(X)(-p))\ar[r]^{\gamma} & \Hodge(H_{2p}(U)(-p)) & 
}
$$
The first part of the lemma follows from the surjectivity of $\alpha$ and $\beta$, and
the five lemma.

We note that
$Gr^W_{-2p}H_{2p-1}(Z)=0$, because $Gr^W_{2p}H^{2p-1}_c(Z)=0$.
Therefore $Gr^W_{-2p}H_{2p}(X)\to Gr^W_{-2p}H_{2p}(U)$ is surjective.
Consequently $\gamma$ is surjective. This implies last part of the lemma.

\end{proof}

Now, let us suppose that $f:X\to Y$ is a surjective morphism with
connected fibres between smooth projective varieties. We suppose also that
$\dim Y<\dim X$. We can analyze
the Hodge structure $H^i(X) $ in terms of the map by applying the decomposition for Hodge
modules \cite{saito}, but we prefer to do this in a more canonical fashion.
First, let us choose a nonempty Zariski open subset  $U\subset Y$, such that $f$
is smooth over $U$, and let $V= f^{-1}U$. 
Then the Leray filtration is defined by
$$L_f^p H^i(V,\Q) =\im [H^i(U,\tau_{\le i-p}\R f_*\Q\to H^i(U,\R f_*\Q)=H^i(V,\Q)]$$
We usually write $L^\dt = L_f^\dt$.
By a theorem of Deligne \cite{deligneL}, the Leray spectral sequence degenerates, so that
$$Gr^b_L H^a(V)\cong H^{b}(U, R^{a-b}f_*\Q)$$
Furthermore, the author \cite[cor 2.6]{arapura} showed that $L$ is a 
 filtration by submixed Hodge structures. Therefore  $Gr^b_L H^a(V) $
 carry natural mixed Hodge structures. These can be seen to coincide with the ones constructed by Saito \cite{saito2}.
  We define a Leray filtration on the Chow group in the most naive way.
Let $L^\dt CH^p(V)$ denote the preimage of $L^\dt H^{2p}(V,\Q)$ under
the cycle map
$$CH^p(V)\to H^{2p}(V,\Q)$$
Then we have an induced cycle map
\begin{equation}
  \label{eq:Box}
   \Box^{p,i}:Gr^i_L CH^p(V) \to H^{i}(U, R^{2p-i} f_*\Q)
 \end{equation}
 Since $\Box^{p,i}(\alpha)\equiv \Box(\alpha) \mod L^{i+1}$, it
 follows that the image lies in
$$\Hodge( H^{i}(U, R^{2p-i} f_*\Q)(p))$$
When $f:X\to Y$ is defined over the above field $K$, one
has an entirely parallel story for $\ell$-adic cohomology. We have a
cycle map
$$\Box^{p,i}:Gr^i_L CH^p(\bar V)\otimes \Q_\ell \to \Tate(H^{i}(\bar
U_{et}, R^{2p-i} \bar f_*\Q_\ell)(p))$$

\begin{lemma}\label{lemma:prod}
  \-
  \begin{enumerate}
  \item If $f_i:V_i\to U_i$ are two smooth projective maps, then
    the exterior product is compatible with $L^\dt$ in the sense that
    $$L_{f_1}^pH^i(V_1,\Q)\otimes L_{f_2}^qH^j(V_2,\Q) \subseteq
    L_{f_1\times f_2}^{p+q}H^{i+j}(V_1\times V_2,\Q)$$
  \item Furthermore
    $$L^r H^{k}(V_1\times V_2,\Q) = \bigoplus_{i+j=k, p+q=r}
    L^pH^i(V_1,\Q)\otimes L^qH^j(V_2,\Q)$$
  \item There is a commutative diagram
    $$   
\xymatrix{
 Gr_L^iCH^p(V_1)\otimes Gr_L^jCH^q(V_2)\ar[r]\ar[d]^{\Box^{p,i}\otimes \Box^{q,j}} & Gr_L^{i+j}CH^{p+q}(V_1\times V_2)\ar[d]^{\Box^{p+q,i+j}} \\ 
 H^i(U_1, R^{2p-i}f_{1*}\Q)\otimes H^j(U_2, R^{2q-j}f_{2*}\Q) \ar[r] & H^{i+j}(U_1\times U_2, R^{2(p+q)-i-j}(f_1\times f_2){*}\Q)
}
    $$
    
    \item When $U=U_1=U_2$, the same statements hold for the fibre
      product $V_1\times_U V_2$.
  \end{enumerate}
\end{lemma}

\begin{proof}
  The first two statements are formal consequences of the isomorphism
  $$(\R f_{1*}\Q, \tau_{\le})\boxtimes (\R f_{2*}\Q, \tau_{\le})\cong
  (\R (f_{1}\times f_2)_*\Q, \tau_{\le})$$
  in the filtered derived category. The third statement is a formal
  consequence of the first two statements, and the compatibility of
  $\Box$ with external products. The fourth follows from the previous
  statements by restricting to the diagonal embedding $U\subset U\times U$.
\end{proof}

Let $n=\dim X$, $m=\dim Y$, and $r= n-m$.
Then the groups on the right of \eqref{eq:Box} will vanish for $(p,i)$ outside the
closed parallelogram $P$ in the $(p,i)$-plane bounded by the lines
  $i=0$, $i=2m$, $i=2p$, and $i=2p-2r$. 
  Let $Q\subset P$ be the closed parallelogram bounded by the lines
  $i=0$, $i=m$, $i=2p$, and $i=2p-r$. This is the lower left quarter
  of $P$.

\begin{thm}\label{thm:leray}
  When $f$ is defined over $\C$, and the cycle map
  $$\Box^{p,i}:Gr^{i}_L CH^p(V) \to \Hodge(H^{i}(U, R^{2p-i} f_*\Q)(p))$$
  is surjective for all $(p,i)$ in $Q$, the Hodge conjecture holds for
  $V$.
   When $f$ is defined over $K$, and the cycle map
  $$\Box_\ell^{p,i}:Gr^{i}_L CH^p(V) \otimes \Q_\ell \to \Tate(H^{i}(\bar U_{et}, R^{2p-i}
  \bar f_{et,*}\Q_\ell)(p))$$
  is surjective for all $(p,i)$ in $Q$, the Tate conjecture holds for $V$.
\end{thm}

\begin{proof}
  We just prove the first statement for the  Hodge conjecture. The  proof of the second statement is almost identical.
  Suppose that we know that $\Box^{p,i}$ is surjective for all
  $(p,i)\in P$.
By lemma~\ref{lemma:HodgeFun}, $\Hodge$ is exact, so
we have a noncanonical isomorphism
\begin{equation}\label{eq:HodgeL}
  \begin{split}
    \Hodge(H^{2p}(V,\Q)(p)) &=\bigoplus_i \Hodge(Gr^i_L(H^{2p}(V,\Q)(p)))\\
    &=\bigoplus_i \Hodge(H^{i}(U, R^{2p-i}  f_*\Q)(p))\\
    &=\bigoplus_i \im Gr_L^iCH^p(V)
  \end{split}
\end{equation}
  This implies  that the Hodge conjecture holds for $V$.

  It remains to show that surjectivity of $\Box^{p,i}$ for $(p,i)\in
  Q$ implies surjectivity for all points in $P$. We do this in two steps. We
  first assume $2p-i\le r$, and then we show that the condition of surjectivity of $\Box^{p,i}$ is
  stable under the reflection  $(p,i)\mapsto (q, i)$, where $q=r+i-p$. 
Let $\OO_Y(1)$ and $\OO_X(1)$ be ample divisors on $Y$ and
$X$ respectively. Cup product with $c_1(\OO_X(1))^{r-2p+i}$  induces an
isomorphism of local systems
$$\xi: R^{2p-i}f_*\Q\stackrel{\sim}{\to} R^{2q-i}f_*\Q$$
Together with lemma~\ref{lemma:prod}, this gives rise to a commutative diagram
  $$
\xymatrix{
 Gr^{i}_L CH^p(V)\ar[r]^>>>{\Box^{p,i}}\ar[d]^{\xi'} & \Hodge(H^{i}(U, R^{2p-i}f_*\Q)(p))\ar[d]^{\xi''} \\ 
 Gr^{i}_L CH^{q}(V)\ar^>>>{\Box^{q,i}}[r] & \Hodge( H^{i}(U, R^{2q-i}f_*\Q))(q))
}
$$
where $\xi'$ and $\xi''$ are   products with
$c_1(\OO_X(1))^{r-2p+i}$. The map $\xi''$ is an isomorphism because
$\xi$ is. If
$\Box^{p,i}$ is surjective, then $\Box^{q,i}$ must also be
surjective. This allows to extend surjectivity of $\Box^{p,i}$ from
$Q$ to the parallelogram  $P'$ bounded by $i=0$, $i=m$, $i=2p$, and $i=2p-2r$.

  To finish the proof, we have to show that if  $\Box^{p,i} $ is
  surjective for $(p,i)\in P'$, then this holds for $P$.
  Let $i\le m$, $\iota = 2m-i$,
and $q=p+m-i$. The transformation $(p,i)\mapsto (q,\iota)$ is a
reflection about $i=m$, which preserves lines of slope $2$.
Therefore $P$ is the union of $P'$ and the image of $P'$ under this reflection.
  If $L$ is a local system defined on $U$, let
$IH^i(Y, L)= H^i(Y, (j_{!*}L[m])[-m])$ denote intersection cohomology
(with the naive indexing convention). 
We have a commutative diagram of vector spaces
$$
\xymatrix{
 IH^{i}(Y, R^{2p-i}f_*\Q)\ar[r]^{\pi}\ar[d]^{\eta'} & Gr^W_{2p} H^{i}(U, R^{2p-i}f_*\Q)\ar[d]^{\eta} \\ 
 IH^{\iota}(Y, R^{2p-i}f_*\Q)\ar[r]^{\pi} & Gr^W_{2q} H^{\iota}(U, R^{2p-i}f_*\Q)
}
$$
where  the vertical maps  are given by the product 
with $c_1(\OO_Y(1))^{m-i}$. The map $\eta'$ is an isomorphism by the hard Lefschetz
for intersection cohomology \cite[6.2.10]{bbd}. The maps labelled by
$\pi$ are the natural ones; these are surjective by \cite{ps} (plus the
comparison theorem in $\ell$-adic case). These
facts imply that $\eta$ is surjective. Now consider the diagram
$$
\xymatrix{
 Gr^{i}_L CH^p(V)\ar[r]^>>>{\Box^{p,i}}\ar[d] & \Hodge(H^{i}(U, R^{2p-i}f_*\Q)(p))\ar[d]^{\eta} \\ 
 Gr^{\iota}_L CH^{q}(V)\ar^>>>{\Box^{q,\iota}}[r] & \Hodge( H^{\iota}(U, R^{2p-i}f_*\Q))(p+m-i))
}
$$
where the vertical maps are again products with
$c_1(\OO(1))^{m-i}$. Since $\Box^{p,i} $ is surjective by assumption, and
$\eta$ is surjective by what we just proved, we can
conclude that $\Box^{q,\iota}$ is surjective. Therefore the surjectivity condition
is preserved by the reflection $(p,i)\mapsto (q,\iota)$, and this
completes the proof.

\end{proof}

\begin{rmk}\label{rmk:leray}
  It is easy to see using a modification of \eqref{eq:HodgeL} that conversely, if the Hodge conjecture holds for
  $V$, then the $\Box^{p,i}$ must surject onto the space of Hodge cycles
  for all $(p,i)\in P$.
\end{rmk}

\begin{cor}\label{cor:leray}
   Suppose that $\dim X=4$, and $\dim Y=m$. If   the cycle map
$$\Box^{2,m}:Gr^{m}_L CH^2(V) \to \Hodge(H^{m}(U, R^{4-m} f_*\Q)(2))$$
is surjective, then the Hodge conjecture holds for  $X$.
In particular, this is the case if
  $\Hodge(H^{m}(U, R^{4-m} f_*\Q)(2))=0$.
\end{cor}

\begin{proof}
  It is well known that the  Hodge conjecture holds for smooth
  projective varieties of dimension at most three (by the Lefschetz
  $(1,1)$ theorem, and the hard Lefschetz theorem). Therefore by
  theorem~\ref{thm:jannsen}, it holds for $X-V$. So by
  lemma~\ref{lemma:localization}, it suffices to prove the conjecture
  for $V$. Applying  theorem \ref{thm:leray}, we have to check the surjectivity for
  $\Box^{p,i}$ for $(p,i)\in Q$. For $p=0$, this trivial, and for $p=1$, this follows from the
  Lefschetz $(1,1)$ theorem. It is easy to check by plotting $Q$, that
  in each of the cases $m=1,2,3$,
  there is exactly one point in $Q$ with $p>1$, namely $(2,1), (2,2),
  (2,3)$ respectively.

\end{proof}

When $U$ is the moduli space of genus two curves with fine level structure, and 
$f:X\to Y$ the universal family of curves, then \cite[cor 3.2]{arapura2} shows that
$$\Hodge(H^{3}(U, R^{1} f_*\Q)(2))\subset  Gr_F^2H^{3}(U, R^{1} f_*\C) =0$$ 
Therefore the Hodge conjecture holds for $X$.
 The corollary also implies a result of Conte and Murre \cite{cm}
 that the Hodge conjecture holds for uniruled fourfolds.
 Similarly, the conjecture holds for fourfolds fibred by surfaces with
 trivial geometric genus:

\begin{cor}
Suppose that  $\dim X=4$, and $\dim Y=2$, and the general fibre of $f:X\to Y$ is a surface with $p_g=0$.
Then the Hodge conjecture holds for $X$.

\end{cor}

\begin{proof}
Given an irreducible component  $T$ of the  relative Hilbert scheme
$Hilb_{X/Y}$ \cite[chap IV]{grothendieck}, let $\supp(T)\subseteq Y$ denote its image in $Y$.
Let
$$C=\bigcup_{\supp(T)\subsetneqq Y} \supp(T)$$
as $T$ ranges over components of $Hilb_{X/Y}$.
Since $C$ is a countable union of proper subvarieties $U-C\not=\emptyset$.  Choose $y\in U-C$. Since $p_g(X_y)=0$, by the Lefschetz $(1,1)$ theorem, there exist
  irreducible divisors $Z_1,\ldots Z_N\subset X_y$ which span $H^2(X_y,\Q)$. Let
  $T_i$ be a component of $Hilb_{X/Y}$ containing $Z_i$. Then $T_i\to Y$
  is surjective. After replacing $T_i$ by an intersection of ample
  divisors we can assume that $T_i\to Y$ is generically finite.
  Let $\cZ_i$ denote the restriction of the universal family over $Hilb_{X/Y}$ to $T_i$.
  Then $\cZ_i$ is a relative divisor on $X/Y$ such that $[\cZ_i]$ is
  equal to a multiple of $[Z_i]$. It follows that, after a finite base
  change, $[\cZ_1], \ldots
  [\cZ_N]$ gives a basis of $R^2f_*\Q$, i.e. an isomorphism
  $\Q(-1)^N\cong R^2f_*\Q$. Therefore, we can identify
  $$H^2(U, R^2f_*\Q)\cong \bigoplus_i H^2(U,\Q)\cup [\cZ_i]$$
  The Lefschetz $(1,1)$ theorem now shows that 
  $$\Hodge(H^{2}(U, R^{2} f_*\Q)(2)) = \bigoplus_i
  \Hodge(H^2(U,\Q)(1))\cup [\cZ_i]$$
  is spanned by algebraic cycles.

\end{proof}

\begin{rmk}
Define the transcendental part $T\subseteq R^2f_*\Q|_U$ to be the smallest sub-variation of Hodge structure such $T_y^{20}= H^{20}(X_y)$ for some (hence all) $y\in U$.
 By essentially the same argument, the Hodge conjecture holds for $X$, if  $\Hodge(H^2(U,T)(2))=0$.
\end{rmk}

\begin{cor}\label{cor:tate}
 Suppose that $f$ is defined over $K\subset \C$, and that $U=Y$. Let $X_\C=X\times_{\Spec K}\Spec \C$ etc.
 If
 $$\Box^{p,i}: Gr^i_LCH^p( X_\C)\otimes \C\to Gr^p_F H^{i}( Y_{ \C}, R^{2p-i} f_{\C,*}\C)$$
 is surjective for all $(p,i)$ in $Q$, then the Hodge and Tate conjectures hold for $X_\C$ and $X$ respectively.
 
\end{cor}

\begin{proof}
 The statement for the  Hodge conjecture is an immediate consequence of  the theorem.
 The assumption implies that  
 $$Gr^p_F H^{2p}(X_\C, \C)= H^p(\bar X, \Omega_{\bar X}^p)\otimes \C$$
 is generated by a finite number of algebraic cycles on $X_\C$, and therefore
 a finite number of cycles on $\bar X$. Tate's conjecture is now a known consequence of Faltings' theorem
 on Hodge-Tate decompositions \cite{faltings}, which implies that
 $$\dim \Tate(H^{2p}(\bar X_{et},\Q_\ell)(p))\le \dim H^p(\bar X, \Omega_{\bar X}^p)$$
 c.f.  \cite[pp 81-82]{tate}.
\end{proof}

\section{Fibre products of elliptic surfaces}

We start by summarizing some facts from Hodge theory needed below. Let $\pi:Y\to \Lambda$ be a smooth  projective curve over a smooth
base $\Lambda$. Let $S\subset Y$ be a divisor \'etale over $\Lambda$, and $U= Y-S$.
We recall that a mixed Hodge module on $Y$, consists a bifiltered regular
holonomic  left $D$-module $(M, F, W)$, a filtered perverse sheaf $(\cL, W)$,
and an isomorphism in the derived category
$$DR(M) = M\to  M\otimes \Omega_Y^1 \to  M\otimes \Omega_Y^2 \ldots \cong \cL\otimes \C$$
compatible with $W$. These are subject to some axioms that we will
not recall \cite{saito, saito2}. The modules of interest to us arise as follows.
Let $(\cL_U,V_U, F_U^\dt,\nabla_U)$ be
 a polarizable variation of  Hodge structure on $U$; here $V_U$ denotes a vector
bundle, $F_U^\dt$ a filtration on it,  $\nabla_U$ an integrable connection, and $\cL$
local system of $\Q$-vector spaces.
By Deligne, we have an extension of $(V_U,\nabla_U)$ to an vector bundle
with logarithmic connection $(V,\nabla)$ on $Y$ such that the  eigenvalues of the residues lie in $[0, 1)$. We filter
this by $F^p V = V\cap j_*F_U^p$.
Then $M=
V(*S)$ is a regular holonomic $D_Y$-module which corresponds under
Riemann-Hilbert to the perverse sheaf $\cL=\R j_*\cL_U[\dim Y]$ tensored with $\C$. 
When equipped with the Hodge filtration
$$F^pM = \sum_i F^iD_Y\cap F^{p+i} V$$
and an appropriate weight filtration $W$, $(M, \cL)$
forms a mixed Hodge module \cite[\S 3]{saito2}.
We have a filtered quasi-isomorphism
\begin{equation}
\label{eq:DRM}
( V\otimes \Omega^\dt_Y(\log S) , F) \cong (DR(M), F)
\end{equation}
where the two sides are filtered by
$$F^p( V\otimes \Omega_Y^i(\log S)) = F^{p-i}V\otimes  \Omega_Y^i(\log S)$$
and
$$F^p( M\otimes \Omega_Y^i ) = F^{p-i}M \otimes \Omega_Y^i$$
Saito \cite{saito, saito2} has shown that mixed Hodge modules are stable under direct images,
and furthermore that   $\R \pi_*$ applied to the second, and therefore also the first,
filtered complex of \eqref{eq:DRM}  is strict. These facts, together
with \cite[thm 0.2]{saito2},  imply that 
$R^1(\pi|U)_* \cL$ is an admissible variation of mixed Hodge structure
such  that
$$Gr^p_F R^1 (\pi_U)_* \cL \otimes \C \cong R^i\pi_*(Gr_F^pV\xrightarrow{Gr\nabla}  Gr_F^{p-1}\otimes \Omega_{Y/\Lambda}^1(\log S))$$
When $\Lambda$ is point, these results go back to  Zucker \cite{zucker}, who showed additionally that in the geometric case,
 the mixed Hodge structure 
on $H^i(U, \cL)$    coincides with the one coming from the Leray spectral sequence.

Given  an admissible  variation of
mixed Hodge structure
 $(V_U, F_U^\dt,\nabla_U,\ldots)$ as above,
we  will describe an associated  ``Kodaira-Spencer" map following the method of Katz and
Oda \cite{ko}. 
It is convenient to pass to the graded 
Higgs bundle $E=\bigoplus E^p$, where  $E^p=Gr_F^pV$, with Higgs field $\theta:E\to E\otimes
\Omega_Y^1(\log S)$ given by the sum of $Gr^p_F(\nabla): Gr^p_FV\to
Gr^{p-1}_FV$,  so that $\theta$ has degree $-1$.   To simplify notation, let us write $\lOmega_Y^p =
\Omega_Y^p(\log S)$ etc.
The pair $(E,\theta)$  is indeed a Higgs bundle in the sense that
$\theta\wedge \theta =0$. Therefore it extends to give a complex
\begin{equation}
  \label{eq:HiggsdeRham}
E\xrightarrow{\theta} E\otimes \lOmega_Y^{1}\xrightarrow{\theta} E\otimes \lOmega_Y^{2}\ldots  
\end{equation}
We define a filtration 
$$K^r(V\otimes \lOmega_Y^{i}) = V\otimes (\lOmega_Y^{i-r}\wedge \pi^*\lOmega_\Lambda^r)$$
One can check that this is a  filtration  by subcomplexes of   the de Rham complex $(V\otimes \lOmega_Y^\dt,\nabla) $.
The connecting map associated to
\begin{equation}
\label{eq:K1K2}
0\to K^{1}/K^{2}\to K^0/K^{2}\to K^0/K^{1}\to 0
\end{equation}
induces  a Gauss-Manin connection
$$\R^1\pi_*(V \otimes \lOmega^\dt_{Y/\Lambda})\cong \R^1\pi_* K^0/K^1\to \R^2\pi_* K^1/K^2\cong
\R^1\pi_*(V\otimes \lOmega^\dt_{Y/\Lambda})\otimes \Omega_\Lambda^1$$
The associated graded of this with respect to $F$ is the Kodaira-Spencer map $\kappa(V)=\sum\kappa_p(V)$.
We can construct $\kappa_p(V)$ directly by applying  $Gr^{p}_F$ to \eqref{eq:K1K2} to obtain
$$ 0\to [0\to E^{p-1}\otimes \lOmega_\Lambda^1\ldots ]\to [E^{p}\to E^{p-1}\otimes \lOmega_Y^1\ldots]\to [E^{p}\to 
E^{p-1}\otimes \lOmega_{Y/\Lambda}^1\ldots]\to 0$$
and then forming the connecting map
$$\R^1\pi_*(E^p\xrightarrow{\theta} E^{p-1}\otimes \lOmega^1_{Y/\Lambda}\ldots)\xrightarrow{\kappa_p(V)}
\R^1\pi_*(E^{p-1}\xrightarrow{\theta} E^{p-2}\otimes \lOmega^1_{Y/\Lambda}\ldots )\otimes \Omega_\Lambda^1$$
Let $\kappa_{p,\lambda}(V)$ denote the fibre of this map at $\lambda$.

We will now describe  the basic set up, that  will be fixed for the remainder of this section
We will be interested in families of elliptic surfaces. More precisely,
let $\E\xrightarrow{f} Y\xrightarrow{\pi} \Lambda$ be a pair of flat projective morphisms, $S\subset Y$ a relative divisor and
$\sigma:Y\to \E$ a section, such that:
\begin{enumerate}
\item $\Lambda$ is smooth.
\item $Y\to \Lambda$ is a smooth relative curve, and the restriction $S\to \Lambda$ is \'etale.
\item For each $\lambda\in \Lambda$, the fibre $f_\lambda:\E_\lambda\to Y_\lambda$ is a relatively minimal nonisotrivial elliptic surface.
\item For each $\lambda\in \Lambda$, $S_\lambda\subset Y_\lambda$ is the discriminant of $f_\lambda$.
\item The data $\E\to Y\supset S$ is topologically locally trivial over
  $\Lambda$. In particular, the sheaves $R^i(\pi|_U)_* (R^1f_*\C)_U$ are locally constant for all $i$, where $U= Y-S$.
\end{enumerate}

The main results of Saito  \cite[thm 0.1]{saito2} shows that $\R (\pi|_U)_* R^1f_*\C$ can be lifted to the derived category
of mixed Hodge modules. Condition (5) above, when combined with \cite[thm 0.2]{saito2}, ensures  that the mixed Hodge module
associated to $R^i(\pi|_U)_* R^1f_*\C$ is an admissible variation of mixed Hodge structure. Choosing $\lambda\in \Lambda$,
we have the Kodaira-Spencer map
\begin{equation}\label{eq:kappa}
\kappa_\lambda = \kappa_{1,\lambda}(R^1f_{\lambda*}\Q): Gr^1_F H^1(U_\lambda, R^1f_{\lambda*}\C)\to
 Gr^0_F H^1(U_\lambda, R^1f_{\lambda*}\C)\otimes \Omega^1_{\lambda}
\end{equation}

Fix $\lambda\in \Lambda$ and an integer $n>0$.
Let $V_\lambda$ be the preimage of $U_\lambda$ in the $n$-fold fibre product 
$T_\lambda=\E_\lambda \times_{Y_\lambda} \E_\lambda\times_{Y_\lambda}\ldots \E_\lambda $.
Suppose that  $f_\lambda$ is semistable,  
which means that the singular fibres  have Kodaira type ${}_1I_b$, i.e. they
are reduced polygons of rational curves.
Then $T_\lambda$ has singularities which are local analytically given by
  $$x_1y_1=x_2y_2=\ldots $$
  This can be resolved by blowing up an explicit monomial ideal
  \cite[lem 5.5]{deligneM} to obtain a nonsingular variety $X_\lambda$.
This is  a special case of   Mumford's toroidal resolution  \cite[p
94]{toroidal}.  We  let $F_\lambda$  denote the projections $V_\lambda\to Y_\lambda$  and $X_\lambda\to Y_\lambda$ (and this should
cause no confusion).

\begin{thm}\label{thm:main}
 With the notation and assumptions as above, suppose furthermore that
 for some
 $\lambda\in \Lambda$, the Kodaira-Spencer map $\kappa_\lambda$, defined in \eqref{eq:kappa}, is injective.
 Then the Hodge conjecture holds for $V_\lambda$, when $\lambda$ is very general.
 In addition, when $f_\lambda$ is semistable,
  the Hodge conjecture holds for the toroidal resolution $X_\lambda$.
\end{thm}

We will give the proof after some preliminaries.
Fix $\lambda\in \Lambda$.
The local system  given by restricting $L= L_\lambda=R^1f_{\lambda*}\Q$ to $U_\lambda$ underlies a polarizable variation
of Hodge structures of type $\{(1,0), (0,1)\}$.
Zucker  \cite{zucker} constructed  mixed Hodge structures  on the various cohomologies
of $L$.
The next lemma shows that the key constituents of these  Hodge structures  coincide. 

\begin{lemma}\label{lemma:compareHS}
  We have isomorphisms
  \begin{equation}\label{eq:IH1}
  H^1(Y_\lambda, L)\cong  IH^1(Y_\lambda,L) \cong  Gr_2^W   H^1(U_\lambda,L)
\end{equation}
 and
   \begin{equation}\label{eq:IH1b}
  Gr^1_F H^1(Y_\lambda, L)\cong  Gr^1_F IH^1(Y_\lambda,L) \cong  Gr^1_F   H^1(U_\lambda, L)
  \end{equation}
  
  \end{lemma}

  \begin{proof}
    For notational simplicity, we write $U$ and $Y$ instead of
    $U_\lambda$ and $Y_\lambda$ in the proof.
    Let $j:U\to Y$ denote the inclusion.
  We have natural morphisms
  $$ L\xrightarrow{\iota} j_*j^* L\xrightarrow{\iota'} \R j_* j^* L$$
  which are clearly isomorphisms over $U$.
  We claim that the first map $\iota$ is an isomorphism  at each $s\in S$. Let
  $s'\in Y$ be a point close to $s$, and $T_s:H^1(X_{s'})\to
  H^1(X_{s'})$ the monodromy about $s$. The map $\iota_s:L_s\to j_*j^*
  L_s$ can be identified with the specialization map $H^1(X_s)\to H^1(X_{s'})^{T_s}$.
  One can check that $\iota_s$ is an isomorphism by proceeding case by case
  through Kodaira tables \cite[p 565, p 604]{kodaira} (or apply semistable reduction to reduce to  checking for $I_n$ type).
  The ``cone" of the second morphism $\iota'$ is  $R^1 j_* j^* L[-1]$,
  which we decompose as $\bigoplus_s V_s[-1]$ where $V_s=H^1(X_{s'})/\im(T_s-1)$. Therefore
  we get an exact sequence
  \begin{equation}
    \label{eq:IHHV}
  0\to IH^1(Y,L) \xrightarrow{\iota'}  H^1(U, L)\to \bigoplus_s V_s    
  \end{equation}
A direct calculation using Kodaira's tables shows that $\dim V_s=1$ when the fibre is of  type $I_n$ (which
is semistable), and zero
in other cases.  To proceed further, we need to recall some of facts about Zucker's mixed Hodge structure
 \cite[\S 13]{zucker}. The lowest weight of $H^1(U, L)$ is $W_2$, and it
 coincides  with image of $\iota$ in \eqref{eq:IHHV}. This is enough
 to prove that  \eqref{eq:IH1} holds.

A careful reading of  \cite[\S 13]{zucker} or
specializing  the formulas on
\cite[pp 515-516]{sz} to the pure case shows that
the spaces $V_s$ are equipped
with  mixed Hodge structures compatible with \eqref{eq:IHHV}. These
Hodge structures are 
constructed  so that for $m>0$,
 $$Gr_{2+m}^W   H^1(U, L) = \bigoplus_s Gr_{2+m}^W V_s
 \cong \bigoplus_s Gr^W_{m} [\Psi_s/( T_s-I) \Psi_s](-1)$$
 where $\Psi_s$ is the   limit mixed Hodge structure associated to $L$
 at $s$.   
This implies that $\dim Gr_{4}^WV_s = \dim V_s=1$
at a semistable fibre.
Therefore for dimension reasons, $V_s=\Q(-2)$. Since $Gr_F^1 \C(-2)=0$,
we get \eqref{eq:IH1b}.
\end{proof}

 Let $M=f_{\lambda*}\omega_{\E_\lambda/Y_\lambda}$.
By  Grothendieck duality, we have a canonical isomorphism $ R^1f_{\lambda*}\OO\cong M^\vee$.
The corresponding graded Higgs bundle $(E,\theta:E\to E\otimes \lOmega^1_{Y_\lambda})$,
is given by $E^0= M^\vee$, $E^1= M$,
$E=E^0\oplus E^1$,
  $$\theta= 
\begin{pmatrix}
 0 & 0\\ \phi & 0
\end{pmatrix}
$$
 where $\phi:E^1\to E^0\otimes \lOmega^1_{Y_\lambda}$ is given by cupping with the
Kodaira-Spencer class associated to $f_\lambda$.
Since this map is assumed to be nonisotrivial, $\phi$ is nonzero.

Let us analyze the tensor power $L^{\otimes N}$, which is a variation of Hodge structure of weight $N$. 
 The corresponding graded Higgs bundle is $(E,\theta)^{\otimes N}$.
In more explicit terms, the underlying vector bundle is a sum of
\begin{equation}\label{eq:EI}
E^I = E^{i_1}\otimes E^{i_2}\ldots= M^{\otimes  2(i_1+i_2+\ldots)-N}
\end{equation}
where $I=(i_1,i_2,\ldots)\in \{0,1\}^{N}$. This is graded by $i_1+i_2+\ldots$.
If $J= (i_i,\ldots, i_k-1, \ldots i_{N})\in \{0,1\}^{N}$, then
let
$$\phi_{IJ}:E^I\to E^J\otimes  \lOmega^1_{Y_\lambda}$$
be given by the tensor product of 
$$\phi:E^{i_k}\to E^{i_k-1}\otimes \lOmega^1_{Y_\lambda}$$
and the identity on the other factors.
Given two vectors $v=(a_1, a_2,\ldots),$ and $w=(b_1, b_2,\ldots)$ in
$\{0,1\}^N$, their Hamming distance $d(v,w)$ is the cardinality of the
set $\{i\mid a_i\not=b_i\}$.
We define $\phi_{IJ}=0$ if $d(I,J)>1$.
The Higgs bundle
\begin{equation}
  \label{eq:tensorHiggs}
(E,\theta)^{\otimes N} = \left(\bigoplus_{I\in \{0,1\}^N}
E^I,\bigoplus_{I,J} \phi_{IJ}\right)
\end{equation}

These formulas extend to  symmetric powers.
We note that the symmetric group on $N$ letters  $S_N$  acts on the
right of $\{0,1\}^N$ in the obvious way, and on the  right on the tensor power by permuting
factors. We have that
\begin{equation}
  \label{eq:permute}
  \begin{split}
     E^I\cdot \sigma &= E^{I\cdot \sigma}\\
\phi_{I\cdot \sigma, J\cdot \sigma} &=\phi_{IJ}\cdot \sigma
  \end{split}
\end{equation}
We define the symmetric power as the invariant part
\begin{equation}
  \label{eq:SNE}
  S^N(E,\theta) = [(E,\theta)^{\otimes N}]^{S_N}\cong \left(\bigoplus_{I\text{ incr.}} E^I,\bigoplus_{I,J \text{ incr.}} \phi_{IJ}\right)
\end{equation}
where $I,J$ are weakly increasing sequences in $\{0,1\}^N$.
 This is the Higgs bundle corresponding  to $S^NL$.

\begin{lemma}\label{lemma:symmetric}
If $\kappa_{1,\lambda}(L)$ is injective, then $\kappa_{p,\lambda}(S^{2p-1}L)$ is injective 
 for all $p>0$.

 \end{lemma}

\begin{proof}
Using \eqref{eq:EI} and  \eqref{eq:SNE}, one finds that
$$Gr_F^p(K^0(S^{2p-1}L)/K^2(S^{2p-1}L))\cong M\xrightarrow{\phi} M^{-1}\otimes \lOmega_Y^1\xrightarrow{\phi} M^{-3}\otimes \lOmega_Y^1\wedge \lOmega_\Lambda^1\ldots$$
and
$$Gr_F^1(K^0(L)/K^2(L))\cong M\xrightarrow{\phi} M^{-1}\otimes \lOmega_Y^1$$
 There is a morphism from the first complex to the second
given by projection, and this preserves the subcomplexes induced by
$K^1$. Therefore, we obtain a commutative diagram with short exact rows
 $$
\xymatrix{
 [0\to M^{-1}\otimes \lOmega_\Lambda^1\to \ldots ]\ar[r]\ar[d] & [M\to M^{-1}\otimes \lOmega_Y^1\to\ldots]\ar[r]\ar[d] & [M\to 
M^{-1}\otimes \lOmega_{Y/\Lambda}^1]\ar[d]^{=} \\ 
 [0\to M^{-1}\otimes \lOmega_\Lambda^1\ ]\ar[r] & [M\to M^{-1}\otimes \lOmega_Y^1]\ar[r] & [M\to M^{-1}\otimes \lOmega_{Y/\Lambda}^1]
}
 $$
 The connecting maps fit into  a commutative diagram
 $$
 \xymatrix{
   Gr_F^p H^1(U_\lambda, S^{2p-1}L)\ar[r]^>>>>>>{\kappa_{p,\lambda}}\ar[d]^{=} & Gr_F^{p-1} H^1(U_\lambda, S^{2p-1}L)\otimes \Omega^1_{\lambda}\ar[d]\\
 Gr_F^1 H^1(U_\lambda, L)\ar[r]^{\kappa_{1,\lambda}} & Gr_F^0 H^1(U_\lambda, L)\otimes \Omega^1_{\lambda}
}
$$
It follows that injectivity of $\kappa_{1,\lambda}(L)$ on the bottom implies
injectivity of  $\kappa_{p,\lambda}(S^{2p-1}L)$ on the top.
\end{proof}

Following standard usage, let us say that a point of $\Lambda$ is very general,
if it lies outside a countable union of proper analytic subvarieties.

\begin{lemma}\label{lemma:Klambda}
  Suppose $(\cL, V,\ldots) $ is an admissible rational variation of mixed  Hodge structure over $\Lambda$
  such that for a very general point $\lambda\in \Lambda$, the Kodaira-Spencer map
  $$K=\kappa_{p,\lambda}(V):Gr_F^p V_\lambda\to Gr_F^{p-1} V_{\lambda}\otimes \Omega^1_\lambda$$
  is injective. Then $V_\lambda$ has no nonzero Hodge cycles of weight $2p$ for very general $\lambda$.
\end{lemma}

\begin{proof}
We can replace $V$ by $Gr_{2p}^W V$ and assume $V$ is pure of weight $2p$.
Fix $\lambda_0\in \Lambda$, and replace $\Lambda$ by a contractible open neighbourhood of $\lambda_0$.
Then (the restriction of) $\cL$ is a constant sheaf of $\Q$-vector spaces. Therefore given $\gamma\in \cL_{\lambda_0}$, it 
extends to section of $\cL$ (denoted by the same symbol) over $\Lambda$.
As $\lambda$ varies, $F^pV_\lambda$ varies holomorphically. Therefore
the locus
$$NL_\gamma = \{\lambda\in \Lambda\mid \gamma_\lambda \text{ is a
  Hodge cycle}\}$$
  is an analytic subvariety of $\Lambda$.  Suppose that $NL_\gamma=\Lambda$ for some
 $\gamma\not=0$.
 Then
 $$K(\gamma)= \nabla(\gamma) \mod F^p =0$$
 This contradicts the hypothesis of the lemma. Therefore $NL_\gamma$ is a proper subvariety.
 Consequently a very general point lies in the complement of $\bigcup_{\gamma\in \cL_{\gamma_0}} NL_\gamma$.
 
\end{proof}

\begin{lemma}\label{prop:H0}
  For any $\lambda$, the cycle map
  $$\Box^{p,0}:Gr^0_LCH^p(V_\lambda)\to H^0(U, R^{2p}F_{\lambda*}\Q)$$
  is surjective.
\end{lemma}

\begin{proof}
  We will suppress $\lambda$ in the argument below.
  Let us write $L=\Q^2$ with the standard action of $SL_2(\Q)$. The
local system $R^1f_*\Q$  can be identified with $L$, with monodromy
given by a homomorphism $\mu:\pi_1(U)\to SL_2(\Z)\subset SL_2(\Q)$. Non-isotriviality
implies that $\im\mu $ is Zariski dense. The K\"unneth formula, allows us
to identify
\begin{equation}
  \label{eq:kunneth}
  R^{k}F_*\Q = \bigoplus_{i_1+\ldots i_n=k}R^{i_1}f_*\Q\otimes
\ldots \otimes R^{i_{n}}f_*\Q
\end{equation}
Let us start with the case of $n=k=2$. Then the decomposition becomes
$$R^2F_*\Q = L^{\otimes 2} \oplus (f_*\Q\otimes R^2f_*\Q)\oplus
(R^2f_*\Q\otimes f_*\Q)$$
So
$$H^0(U, R^2F_*\Q) = H^0( L^{\otimes 2}) \oplus H^0(f_*\Q\otimes R^2f_*\Q)\oplus
H^0(R^2f_*\Q\otimes f_*\Q)\cong \Q^3$$
Let $D\subset \E$ denote a section of $f$, and let $\Delta\subset \E\times_Y \E$ the image of $\E$
under the diagonal embedding. 
Since $[\Delta], p_2^*[D], p_1^*[D]\in H^0(R^2F_*\Q) $ are easily seen to be linearly
independent, they necessarily span this space. A linear  combination of these
classes give a generator 
\begin{equation}\label{eq:Delta0}
\Delta_0= [\Delta]-p_1^*[D]-p_2^*[D]\in
H^0(L^{\otimes 2})=H^0(\wedge^2 L)\cong \Q
\end{equation}

We now turn to the general case. Let $p_i:T\to \E$, and $p_{ij}:T\to
\E\times_Y \E$ also denote the projections
onto the $i$th and $ij$th factors.
When $k=2p$, equation \eqref{eq:kunneth} can be written as
$$R^{2p}F_*\Q = \bigoplus_{q\le p}\bigoplus_\sigma
\sigma((f_*\Q)^{\otimes n-2q-(p-q)}\otimes L^{\otimes 2q}\otimes
(R^2f_*\Q)^{\otimes (p-q)})$$
where $\sigma$ runs over a suitable set of permutations depending on $q$.
It follows that
\begin{equation}
  \label{eq:kunneth2}
  H^0(U, R^{2p}F_*\Q)= \bigoplus_q \bigoplus_\sigma  H^0(U, L^{\otimes
  2q})\otimes H^0(U, R^2f_*\Q)^{\otimes (p-q)}
\end{equation}
Invariant theory \cite[prop F.13]{fh} tells us that $H^0(L^{\otimes 2q})$
is generated by products of  pullbacks of  classes from
$H^0(L^{\otimes 2})$ under various projections $p_{ij}$. This means
that $H^0(U, L^{\otimes 2p})$ is spanned by products of
$p_{ij}^*\Delta_0$.
Therefore the summands on the right of \eqref{eq:kunneth2}
 are generated by classes  of the form $p_{ij}^*\Delta_0$ and $p_k^*D$.
This proves that $H^0(U, R^{2p}F_*\Q)$ is spanned by algebraic cycles.

\end{proof}

\begin{lemma}\label{lemma:Hodge0}
  For very general $\lambda$,
   $$\Hodge(H^1(U_\lambda, R^{2p-1}F_{\lambda*}\C)(p))=0$$
 \end{lemma}

 \begin{proof}
We will suppress $\lambda$ again.
   We can view $GL_2(\Q)$ as the (big) Mumford-Tate group of the
variation of Hodge structure $L$, or equivalently of a very
general fibre of $L$. We have that  $S^nL(i)= S^nL\otimes (\det
L)^{-i}$ is also a $GL_2$-module. We can decompose
\begin{equation}\label{eq:decomp}
L^{\otimes 2p-1} = S^{2p-1} L \oplus S^{2p-3}L(-1)^{\oplus
  m_1}\oplus\ldots
\end{equation}
into a sum of irreducible $GL_2$-modules.  Lemmas
\ref{lemma:symmetric} and \ref{lemma:Klambda} shows
that
\begin{equation*}
  \begin{split}
    \Hodge(H^1( S^{2p-1} L)(p)) &= 0\\
    \Hodge( H^1( S^{2p-3} L(p-1))&=0\\
    \ldots&
  \end{split}
\end{equation*}

\end{proof}

\begin{proof}[Proof of theorem \ref{thm:main}] Choose $\lambda$ very general.
Since $f_\lambda$ is nonisotrivial, $U_\lambda\subsetneqq Y_\lambda$. Therefore $H^i(U_\lambda, R^j f_{\lambda*}\Q)=0$ when $i>1$.
  The theorem is now a consequence of
  theorem~\ref{thm:leray},  and  lemmas \ref{lemma:localization},
 \ref{prop:H0} and  \ref{lemma:Hodge0}.
\end{proof}

We will now discuss a  few corollaries. 

\begin{cor}
 If $\lambda$ is very general, the Hodge conjecture holds for any desingularization of $\E_\lambda\times_{Y}\E_\lambda\times_{Y}\E_\lambda$.
\end{cor}

\begin{proof}
This follows from the theorem, remark \ref{rmk:leray}, and corollary \ref{cor:leray}.
\end{proof}

Now, let us assume that
$\Lambda =\{\lambda\}$
is a point. Then $f:\E\to Y$ is an elliptic surface.
In this case, the  assumption  that $\kappa_\lambda$ is
injective is equivalent to
$$ Gr^1_F H^1(U, R^1f_*\C)=0$$
This condition  occurs in the literature in  a different form.
 An elliptic surface $f:\E\to Y$  is said to be extremal (c.f. \cite[p 75]{miranda})  if the Picard number $\rho$
  equals the Hodge number $h^{11}$, and the Mordell-Weil group $MW(\E/Y)$
  (which is the group of sections) has  rank $0$. 

  \begin{lemma}\label{lemma:extremal}
  We have
  \begin{equation}\label{eq:IH11st}
\dim Gr^1_F H^1(U, R^1f_*\C)= h^{11}(X)-\rho(X)+ \rank MW(\E/Y) 
\end{equation}
    The surface is  extremal if and only if
  $$Gr^1_F H^1(U, R^1f_*\C)=0$$
  \end{lemma}

  \begin{proof} 
   By Saito's version of the
decomposition theorem \cite[p 857]{saito}, we can decompose $\R f_*\Q$ 
as a sum of intersection cohomology complexes up to shift in the
constructible derived category, and moreover these complexes underly
pure Hodge modules. 
By restricting this sum to $U$,  we can identify some of 
these components explicitly:
\begin{equation}
  \label{eq:decompA}
\R f_*\Q \cong  \underbrace{\Q}_{f_*\Q}\oplus  j_*j^*R^1f_*\Q[-1]\oplus  \underbrace{\Q}_{j_*j^*R^2f_*\Q}[-2]\oplus
  M
\end{equation}
The, as yet undetermined, term $M$  is supported on the finite set $S$.
This yields a (noncanonical) decomposition
\begin{equation}
  \label{eq:decompB}
H^{2}(X,\Q) \cong   f^*H^2(Y,\Q)\oplus IH^1(R^1f_*\Q)\oplus H^0(Y, j_*j^*R^2f_*\Q)\oplus H^2(M)  
\end{equation}
 The first summand on the
right is spanned by the fundamental class $[X_t]$ of a fibre. The
third summand is spanned by the class $[\sigma]$.
To calculate $M$, we restrict to a point $s\in S$, and observe that $H^*(\R
f_*\Q|_s)$ is the cohomology of the fibre $X_s=f^{-1}(s)$ by proper
base change \cite[p 41]{dimca}. Therefore $M$
gives the excess cohomology not coming from the   preceding terms in \eqref{eq:decompA}. Let
$$X_s=\sum_{i=1}^{m_s} n_{s,i}X_{s,i}$$
be the decomposition into irreducible components. Let $D_s$ be a small disk
centered at $s$, $t\in D_s^*=D_s-\{s\}$, and $\gamma_s\in
\pi_1(D_s^*,t)$ a generator. We see that
$$H^i(X_s,\Q)=
\begin{cases}
  \Q  &\text{if $i=0$}\\
H^1(X_t,\Q)^{\gamma_s} \cong (j_*j^*R^1f_*\Q)_s&\text{if $i=1$}\\
  \Q^{m_s}  &\text{if $i=2$}\\
  0 &\text{otherwise}
\end{cases}
$$
Therefore $M \cong \bigoplus \Q_s^{m_s-1}[-2]$.
From \eqref{eq:decompB}, we deduce that we have
a noncanonical decomposition
\begin{equation}
  \label{eq:decompC}
H^{2}(X) \cong   IH^1(R^1f_*\Q)\oplus \Q[\sigma]\oplus\Q[X_t]  \oplus \bigoplus_{s} \Q^{m_s-1}  
\end{equation}
We can see that the last two summands are spanned by divisor classes
supported on the fibres. Since the decomposition \eqref{eq:decompA}
can be lifted to the derived category of mixed Hodge modules, we see that \eqref{eq:decompC} becomes
an isomorphism of Hodge structures provided that all the summands in \eqref{eq:decompC}  except the
first are viewed as sums of  the  Tate structures $\Q(-1)$.

 Formula \eqref{eq:decompC} implies
 \begin{equation}\label{eq:IH11}
\dim Gr^1_F IH^1(R^1f_*\C)= h^{11}(X) - 2-\sum (m_s-1)
\end{equation}
When combined  with  lemma \ref{lemma:compareHS} and the Shioda-Tate formula
\cite{miranda}
\begin{equation}\label{eq:st}
\rank MW(\E/Y) = \rho(X) - 2-\sum (m_s-1)
\end{equation}
we obtain \eqref{eq:IH11st}.
The second part of the lemma characterizing extremal surfaces follows
immediately from this and lemma~\ref{lemma:compareHS}.
  \end{proof}

 A key set of examples is provided by:
  
  \begin{prop}[Shioda]
    Elliptic modular surfaces are extremal.
  \end{prop}

  \begin{proof}
   The   Mordell-Weil rank is zero by \cite[thm 5.2]{shioda}, and
   $\rho=h^{11}$ by \cite[rmk, 7.8]{shioda}.
  \end{proof}

  We can now  state the first corollary of theorem \ref{thm:main}.
This corollary applies to elliptic modular surfaces, so
we recover Gordon's theorem \cite{gordon}.

\begin{cor}\label{cor:extremal}
If $f:\E\to Y$ is an extremal   elliptic surface 
  then the Hodge conjecture holds for the complement $V$ of the
  preimage of $S$ in $\E\times_{Y} \E\times_{Y}\ldots  \E$.; it holds for 
  any desingularization of $\E\times_{Y}\E \times_{Y}\E$, and it holds for the toroidal
  resolution $X$ when $f$ is semistable, 
\end{cor}

\begin{proof}
  When $\Lambda=\{\lambda\}$, $\lambda$ is very general since
  $\emptyset$ is the only proper analytic subset.
\end{proof}

  
We next consider   families of  elliptic K3 surfaces with the following features:
  $$\E\xrightarrow{f}\PP^1\times \Lambda=Y\to \Lambda$$
 is  a family of elliptic K3 surfaces over a curve $\Lambda$ such that:
  \begin{enumerate}
\item[(EK1)] $\rank MW(\E_\lambda/\PP^1)=0$ for very general $\lambda$.
\item[(EK2)] There is a fibrewise isogeny between $\E$ and the family of Kummer surfaces
  associated to a selfproduct of a nonisotrivial elliptic surface
  $g:F\to\Lambda$. More precisely, there is an
  algebraic correspondence
  $$Z\in CH^2(\E\times   Km(F\times_\Lambda F/\Lambda))$$
  which  induces a fibrewise  isomorphism between the transcendental parts
  \begin{equation}\label{eq:tr}
H_{tr}^2(\E_\lambda,\Q)\cong H_{tr}^2(Km(F_\lambda\times F_\lambda),\Q)\cong  H_{tr}^2(F_\lambda\times F_\lambda,\Q) 
\end{equation}
as Hodge structures. (Recall that the  transcendental part $H^2_{tr}(\E,\Q)=(NS(\E)\otimes \Q)^\perp$.)
\end{enumerate}
 Examples satisfying (EK1) and (EK2) include Hoyt's  \cite{hoyt} defined by the affine equation
  $$y^2 = t(t-1)(t-\lambda)x(x-1)(x-t)$$
  For $\lambda \in \Lambda= \C-\{0,1\}$, this defines an elliptic K3 surface $\E_\lambda$ over the $t$-line. 
  Hoyt shows that this is birational to the quotient of an explicit Kummer surface $Km(F_\lambda \times F_\lambda)$ by an involution, and that the   rank of the Mordell-Weil group $MW(\E_\lambda/\PP^1)$ is zero unless
  $F_\lambda$ has complex multiplication. 
  
  The families of K3 surfaces constructed by Dolgachev
\cite[Ex. 7.8, 7.9, 7.10]{dolgachev} can also be seen to  satisfy (EK1) and (EK2) above. The first condition is satisfied by construction (see the
remarks preceding \cite[Ex. 7.8]{dolgachev});
for the second use \cite[thm 7.6]{dolgachev}.

\begin{cor}[to theorem \ref{thm:main}]
Let $\E\to \PP^1\times \Lambda$ be a family satisfying (EK1) and (EK2). Suppose that  $\lambda$ is very general.
 The Hodge conjecture holds for $V_\lambda$, and for any desingularization of $\E_\lambda\times_{\PP^1}\E_\lambda\times_{\PP^1}\E_\lambda$.
\end{cor}

\begin{proof}
 The correspondence $Z$ induces an inlcusion
 $$H_{tr}^2(\E_\lambda,\Q)\cong  H_{tr}^2(F_\lambda\times F_{\lambda},\Q) \subseteq S^2H^1(F_{\lambda},\Q)$$
Let us choose $\lambda$  very general. Then $F_\lambda$ will not have
complex multiplication, which implies that  $S^2H^1(F_{\lambda})$ will be
a simple Hodge structure.  Simplicity forces the  above inclusion to be an
equality.  It follows that
  $\dim H_{tr}^2(\E_\lambda,\Q)=3$, and therefore
  $\rho(\E_\lambda)=19$.
Lemma \ref{lemma:extremal} shows that
  $\dim Gr_F^1 IH^1( R^1 f_{\lambda*}\C)= 1$. The orthogonal complement of
  $IH^1( R^1 f_{\lambda*}\C)$ is spanned by divisors by \eqref{eq:decompC}.
  Therefore  $IH^1( R^1 f_{\lambda*}\C)$ must 
  contain the $(2,0)$ and $(0,2)$ parts of $H^2(\E_\lambda)$. Consequently its dimension is $3$ and
  it meets $H_{tr}^2(\E_\lambda)$. By simplicity of the latter, we
  must have that
  $$IH^1( R^1 f_{\lambda*}\Q) = H_{tr}^2(\E_\lambda)$$
  Therefore by lemma~\ref{lemma:compareHS} we have an isomorphism
  $$Gr_2^WH^1(U_\lambda, R^1 f_{\lambda*}\Q) \cong S^2H^1(F_{\lambda},\Q)$$
  induced by $Z$, and this leads to an isomorphism of
  corresponding variations of Hodge structure, possibly over a
  nonempty Zariski open
  subset of $\Lambda$. Therefore it suffices to  check injectivity of
  the Kodaira-Spencer map
  $$ \kappa_{2,\lambda}(S^2R^1g_*\Q):Gr^2_F S^2H^1(F_{\lambda})\to Gr_F^1 S^2H^1(F_{\lambda})$$
  Using \eqref{eq:SNE} and the formulas preceding it, we can see that
  $ \kappa(S^2R^1g_*\Q )$ is the tensor product 
  $$\kappa_{1,\lambda}(R^1g_*\Q)\otimes id:H^0(F_\lambda,\Omega^1)\otimes H^1(F_\lambda,\OO)\to H^1(F_\lambda,\OO)\otimes H^1(F_\lambda,\OO)$$
 and this is injective for general $\lambda$, because $g$
 is nonisotrivial.

\end{proof}

\begin{rmk}
The conditions  (EK1) and (EK2) force $\rho(\E_\lambda)=19<h^{11}$ for very general
$\lambda$, so these surfaces cannot be extremal.  Conversely, using lemma \ref{lemma:extremal},
we can see that any nontrivial family of elliptic
K3 surfaces to which the theorem applies must have $\rho(\E_\lambda)=19$ generically.
\end{rmk}

\section{Families of abelian varieties over a Shimura curve}

For our purposes, a Shimura variety $Y$ is a quotient of a Hermitian symmetric domain by an arithmetic group.
It can constructed from the following datum:  a  semisimple algebraic group $G$ over $\Q$,
a  homorphism (or conjugacy class of homomorphisms) $h:U(1)\to G^{ad}(\R)$  satisfying the axioms of \cite[4.4]{milne},
and an arithmetic group $\Gamma\subset G(\Q)$. (The group $\Gamma$ is usually required to be a congruence group,
but we will not insist on this.)
 If  $K$ is the stabilizer of $h$ in the identity component  $G(\R)^+$, the corresponding Shimura variety
$Y=\Gamma\backslash G(\R)^+/K$.
For example, given a rational symplectic vector space $(V,Q)$ of dimension $2g$,
let $G=Sp(V,Q)$. We have a Shimura datum, where the set of homomorphisms $h$ corresponds to
 the set of Hodge structures of type $\{(-1,0), (0,-1)\}$ on $V$, polarized by $Q$.
 If we fix a lattice $L\subset V$, such that $Q|_L$ is principal, and let  $\Gamma\subset G(\Q)$, be the subgroup preserving $L$,
 the corresponding Shimura variety is just the moduli space of principally polarized abelian varieties $A_g$.
We say that a Shimura variety $Y$ is of Hodge type if there exists a faithful representation $G\to Sp_{2g}(\Q)$
which extends to a map of Shimura data.  
Then we get an embedding of $Y$ into the moduli space of abelian varieties with level structure.
If level structure is fine, then we can pull back the universal family to get a family of abelian
varieties over $Y$.

%
%

We now describe the main example of interest to us.
Let $K/\Q$ be a totally real number field of degree $d$.
Denote the $d$ distinct embeddings  by $\sigma_i:K\to \R, i=1,\ldots d$.
Fix a quaternion division algebra $D/K$ such  that $D\otimes_{\sigma_1}\R\cong M_2(\R)$
and  $D\otimes_{\sigma_i}\R$ is the Hamilton quaternions $H$ for $i=2,\ldots, d$.
If $\bar x$ denotes conjugation in $D$, then 
$$G=D_1^*=\{x\in D\mid x\bar x=1\}$$
is an algebraic group over $\Q$ whose group of real points  
\begin{equation}\label{eq:D1R}
G(\R)\cong SL_2(\R)\times (SU_2)^{d-1}
\end{equation}
It follows that
\begin{equation}\label{eq:D1C}
G(\C) \cong SL_2(\C)^{d}
\end{equation}
and 
\begin{equation}\label{eq:D1ad}
G^{ad}(\R) \cong PSL_2(\R)\times (PU_2)^{d-1}
\end{equation}
Let $h:U(1)\to G^{ad}(\R)$ be defined by
$$h(e^{i \theta}) = 
\left( \begin{pmatrix}
 \cos\theta/2 &\sin\theta/2\\ -\sin\theta/2 & \cos\theta/2
\end{pmatrix}\mod \pm I, 1\right)
$$
with respect to \eqref{eq:D1ad}.  (For an explanation of $\theta/2$, see \cite[ex 1.10]{milne}.)
The pair $(G,h)$ defines a connected Shimura datum.
The corresponding symmetric space is the upper half plane $\HH$.
To complete the description of the Shimura variety,
 fix a maximal order $\OO$ in $D$,
and a torsion free finite index subgroup of $\Gamma\subseteq \OO_1^*=\OO\cap D_1^*$.
Let $Y= \Gamma\backslash \HH$, where $\Gamma$
acts through projection to $SL_2(\R)$.
{\em We fix this notation for the remainder of this section.}

We first observe  that
$Y$ is a projective algebraic curve, c.f. \cite[chap 9]{shimura}.  We claim that 
the curve $Y$ is of Hodge type, so it would  carry a
universal family of abelian varieties $f:\cA\to Y$. 
An explicit construction of  $\cA$ was given in a  paper of Viehweg-Zuo \cite{vz}.
We outline the construction from section 5 of their paper,
because we will need to extract certain facts which were not explicitly stated.
The correstriction of $D$ is the central simple $\Q$-algebra
$$\Cor(D) = [\bigotimes_i
D\otimes_{\sigma_i}\overline{\Q}]^{Gal(\overline{\Q}/K)}$$
of dimension $4^d$.
In the simplest case, $\Cor(D)$ splits, which means that it is a matrix algebra
$M_{2^d}(\Q)$. Then via the norm homomorphism $Nm:D^*\to \Cor(D)^*$, $V= \Q^{2^d}$ is a representation of $G$.
If $\Cor(D)$ does not split, then it will split over a suitable quadratic field $\Q(\sqrt{b})$.  In this case, 
$V=\Q(\sqrt{b})^{2^d}$ becomes a representation of $G$.
Under \eqref{eq:D1R}, $V\otimes_\Q \R\cong \R^2\otimes_\R V'$, where $\R^2$ is the standard representation
of $SL_2(\R)$, and the second factor 
$$
V'=
\begin{cases}
 \R^{2^{d-1}} &\text{if $\Cor(D)$ splits}\\
 \R^{2^d} &\text{otherwise}
\end{cases}
$$
is a representation of $SU_2^{d-1}$. The representation $V'$  has no trivial factors,
so it carries an $SU_2^{d-1}$-invariant inner product $Q_2$,
which unique up to a scalar factor.
The complex representation
$$
V'\otimes \C\cong
\begin{cases}
 V_2\otimes\ldots \otimes V_{d-1} &\text{if $\Cor(D)$ splits}\\
(V_2\otimes\ldots \otimes V_{d-1})^{\oplus 2}  &\text{otherwise}
\end{cases}
$$
where $V_i$ is the standard representation of $SL_2(\C)$ acting through the $i$th factor of \eqref{eq:D1C}.
Let $Q_1$ denote the standard symplectic form on $\R^2$.
A multiple of $Q_1\otimes Q_2$ will be a $G$-invariant $\Q$ or $\Q(\sqrt{b})$-valued symplectic form  $Q_3$ on $V$. 
Setting $Q=Q_3$ in the first case, or $Q=\tr_{\Q(\sqrt{b})/\Q} \circ Q_3$ in the second gives an invariant $\Q$-valued form.
Hence, we have an injective homomorphism
of algebraic groups $\rho:G\to Sp(V,\Q)$. This  extends to a map of Shimura data, so $Y$ is of Hodge type. Note that  $Sp(V,\Q)$ is
isomorphic to  $Sp_{2^d}(\Q)$ or $Sp_{2^{d+1}}(\Q)$ depending on the cases.
We summarize this discussion, along with some additional facts from \cite[lemma 5.8, p 272]{vz}. 

\begin{thm}[Viehweg-Zuo]\label{thm:vz}
There is an abelian scheme $f:\cA\to Y$ such that:
\begin{enumerate}
\item The relative dimension of $\cA$ is either $2^{d-1}$, if $\Cor(D)$
  splits, or $2^d$ otherwise.
\item The local system $R^1f_*\Q$ corresponds to a representation 
$$
\xymatrix{
 \Gamma\ar[r]^{\rho}\ar[d] & GL(V) \\ 
 G\ar[ru] & 
}
$$
such  that:
\begin{enumerate}
\item The representation of $G(\R)$ on $V\otimes \R$ is a tensor product $\R^2\otimes V'$,
where $\R^2$ is the standard representation of $SL_2(\R)$, and $V'$ is an orthogonal representation of the remaining
factor of \eqref{eq:D1R}.
\item The representation of $G(\C)$ on $V\otimes \C$ is isomorphic to a sum of one (if $\Cor(D)$ splits) or two
copies (otherwise) of $V_1\otimes\ldots\otimes V_d$, where $V_i$ is the standard representation of the $i$th  factor of \eqref{eq:D1C}.
\end{enumerate}

\item The Zariski closure of $\rho(\Gamma)$ and the special Mumford-Tate group of a very general fibre of $f$ are both equal to $G$.

\item Corresponding to the decomposition in (2b), the graded Higgs bundle associated to the variation of Hodge structure $R^1f_*\C$ 
  is a tensor product
  $$(E_1^{0}\oplus E_1^{1},\theta)\otimes (E_2^0,0)\otimes\ldots \otimes (E_d^0,0)$$
  Furthermore, $\theta:E_1^{1}\to E_1^{0}\otimes \Omega_Y^1$ is an
isomorphism.
\end{enumerate}

\end{thm}

Let $F:X\to Y$ be the $n$-fold fibre product $\cA\times_Y\ldots \times_Y \cA$.
 We note that $\cA\to Y$, and therefore $F$, is defined over a number field $k$.
 This follows from the theory of canonical models \cite{milne}, although we should
 point out that since we are working with connected Shimura varieties, $k$ might be bigger than the reflex field.

\begin{thm}\label{thm:SCHodge}
  The Hodge and Tate conjectures holds for $X$.
\end{thm}

The proof follows the same basic strategy as the proof of
theorem~\ref{thm:main}. We first establish the preliminary results.

\begin{lemma}\label{lemma:inv}
  Let $V_i$ be standard representations of $SL_2(\C)$, and let
  $V=V_1\otimes \ldots \otimes V_d$ be viewed as an
  $H=SL_2(\C)^d$-module. Then invariant part of the tensor algebra
$$(T^*V)^H$$
  is  generated by $(V\otimes V)^H$.
  
\end{lemma}

\begin{proof}
  We have an isomorphism
  $$(T^*V)^H= (T^*V_1)^{SL_2(\C)}\otimes \ldots  \otimes (T^*V_d)^{SL_2(\C)}$$
  So the result follows from classical invariant theory \cite[prop F.13]{fh}.
\end{proof}

\begin{prop}
The  cycle maps
  $$\Box^{p,i}: Gr^i_L CH^p(X)\to H^i(Y, R^{2p-i}F_*\Q)$$
  are surjective for $i=0,2$ and all $p$.
\end{prop}

\begin{proof}
  We first treat the case of $i=0$. The surjectivity of $\Box^{0,0}$
  is trivially true, and  the surjectivity of  $\Box^{1,0}$
  is a consequence of the Lefschetz $(1,1)$ theorem. We have that
  $R^if_*\Q= \wedge^i R^1f_*\Q$. This together with the  K\"unneth
  formula and lemma \ref{lemma:inv} implies that the product map
  $$S^p H^0(Y, R^2F_*\Q)\to H^0(Y, R^{2p}F_*\Q)$$
  is surjective. Therefore
  $$ S^p Gr^0_L CH^1(X)\to H^0(Y, R^{2p}F_*\Q)$$
  is surjective. This implies the surjectivity of $\Box^{p,0}$.

 We have the  commutative diagram
  $$
\xymatrix{
 Gr^0_L CH^p(X)\ar[r]^{\Box^{p,0}}\ar[d] & H^0(Y, R^{2p}F_*\Q)\ar[d]^{\cong} \\ 
 Gr^2_L CH^{p+1}(X)\ar[r]^{\Box^{p+1,2}} & H^2(Y, R^{2p}F_*\Q)
}
$$
where the isomorphism on the right is hard Lefschetz.
This implies surjectivity of $\Box^{p+1,2}$.

\end{proof}

\begin{prop}
  For all $p$,
  $$Gr^p_FH^1(Y, R^{2p-1}F_*\C)=0$$
\end{prop}

\begin{proof}
The local system  $ R^{2p-1}F_*\C$ corresponds to the $SL_2(\C)^d$-module $\wedge^{2p-1}(V^{\oplus n})$. 
We can decompose it into a sum of irreducible modules, and reduce to the problem of showing that
\begin{equation}\label{eq:H1D0}
H^1(M^q\xrightarrow{\phi} M^{q-1}\otimes \Omega_Y^1)=0
\end{equation}
where
$$(M,\phi)= S^{2q-1}(E_1,\theta)\otimes  S^{r_2}(E_2,0)\otimes\ldots S^{r_d}(E_d,0)$$
$$= S^{2q-1}(E_1,\theta)\otimes (M',0)$$
and $q,r_1,\ldots$ are arbitrary.
Let $I= (1,\ldots, 1,0\ldots, 0)$ with $q$ $1$'s followed by $q$ $0$'s.
Then  we have an isomorphism of complexes
$$M^q\xrightarrow{\phi} M^{q-1}\otimes \Omega_Y^1\cong [E_1^1\xrightarrow{\sim} E_1^0\otimes\Omega_Y^1]\otimes E_1^I\otimes M'$$
This implies \eqref{eq:H1D0}.
 
\end{proof}

\begin{proof}[Proof of theorem~\ref{thm:SCHodge}]
This is a consequence of corollary \ref{cor:tate} and the last two propositions.  
\end{proof}


\begin{thebibliography}{999}

 

\bibitem[A1]{arapura} D. Arapura, {\em The Leray spectral sequence is
    motivic} Invent. Math (2005)
\bibitem[A2]{arapura2} D. Arapura, {\em Algebraic cycles on genus-$2$
    modular fourfolds} Alg. Num. Theory (2019)

  

\bibitem[B]{beilinson} A. Beilinson, {\em Notes on absolute Hodge
    cohomology}, Contemp. Math., 55, AMS, (1986)

\bibitem[BBD]{bbd} A. Beilinson, J. Bernstein, P. Deligne, {\em
    Faisceux pervers}, Ast\'erisque 100 (1982)
    
    \bibitem[CM]{cm} A. Conte, J. Murre, {\em The Hodge Conjecture for fourfolds admitting a covering by rational curves.}
     Math. Ann. 238, 79 - 88 (1978)

\bibitem[D1]{deligneL} P. Deligne, {\em Theoreme de Lefschetz et crit\'eres de d\'eg\'en\'escence de suites spectrales}, Pub IHES (1968)

\bibitem[D2]{deligneM} P. Deligne, {\em Formes modulaires et repr\'esentations $\ell$-adiques} Sem. Bourbaki (1969)

\bibitem[D3]{deligneH} P. Deligne {\em Th\'eorie de Hodge II}
  Pub. IHES  (1971)

  \bibitem[Di]{dimca} A. Dimca, {\em Sheaves in topology}, Springer
    (2004)

    \bibitem[Do]{dolgachev}  
I. Dolgachev, {\em Mirror symmetry for lattice polarized K3 surfaces. }
Algebraic geometry, 4. 
J. Math. Sci. 81 (1996), no. 3, 2599–2630

\bibitem[F]{faltings} G. Faltings, {\em p-adic Hodge theory}, JAMS 1, (1988)
  
\bibitem[F1]{fulton} W. Fulton, {\em Intersection theory}, Springer (1984)

\bibitem[F2]{fultonT} W. Fulton, {\em Introduction to toric varieties. }Annals of Mathematics Studies, 131.
Princeton University Press, (1993)
  
\bibitem[FH]{fh} W. Fulton, J. Harris, {\em Representation theory},
  Springer (1991)

  
\bibitem[G]{gordon} B. Gordon, {\em Algebraic cycles and the Hodge
    structure of a Kuga fiber variety} TAMS (1993)
    
    
\bibitem[Gr]{grothendieck} A. Grothendieck, {\em Fondements de
    g\'eom\'etrie alg\'ebrique}, Sem. Bourbaki 1957-1962, (1962)

  \bibitem[H]{hoyt} W. Hoyt, {\em Notes on elliptic K3 surfaces.}
  Number theory (New York, 1984–1985), 196–213, Lecture Notes in Math., 1240, Springer, Berlin, 1987. 
  
  
\bibitem[J]{jannsen} U. Jannsen, {\em Mixed motives and algebraic
    K-theory}, LNM 1400 Springer (1990)

    \bibitem[KO]{ko} N. Katz, T. Oda, {\em On the differentiation of De
          Rham cohomology classes with respect to parameters}
        
\bibitem[KKMS]{toroidal}  G. Kempf, F. Knudsen, D. Mumford, and
  B. Saint-Donat, {\em Toroidal embeddings. I.} Lecture Notes in Mathematics, 339, Springer-Verlag, Berlin-New York, (1973)
  
  \bibitem[Kd]{kodaira} K. Kodaira, {\em On compact analytic surfaces. II, III. } Ann. of Math. (2) 77 (1963), 563–626; ibid. 78 1963 1–40.

    
    \bibitem[L]{lewis} J. Lewis, {\em A survey of the Hodge conjecture. Second edition. Appendix B by B. Brent Gordon.} CRM Monograph Series, 10. AMS (1999)
    
    \bibitem[Mi]{milne} J. Milne, {\em Introduction to Shimura varieties}, Harmonic analysis, the trace formula, and Shimura varieties, 265–378, 
      Clay Math. Proc., 4,  AMS (2005)

      \bibitem[Mr]{miranda} R. Miranda, {\em The basic theory of
          elliptic surfaces.} ETS Editrice, Pisa, (1989)

  
  \bibitem[M]{mumford} D. Mumford, {\em A note of Shimura’s paper ``Discontinuous groups and Abelian varietes”,} Math. Ann. 181 (1969) 345–351

    

\bibitem[PS]{ps} C. Peters, M. Saito, {\em Lowest weights in
    cohomology of variations of Hodge structure}, Nagoya Math 206
  (2012)

  
\bibitem[Sa1]{saito} M. Saito, {\em Modules de Hodge polarisibles}, Publ. Res. Inst. Math. Sci. 24, 849–995
(1988)

\bibitem[Sa2]{saito2} M. Saito, {\em Mixed Hodge modules }, Publ. Res. Inst. Math. Sci. 
(1990)
  
\bibitem[Sc]{schmid} W. Schmid,  {\em Variation of Hodge structure},
  Invent. Math. 22 (1972)
  
  
  \bibitem[Sh]{shimura} G. Shimura, {\em Introduction to the Arithmetic Theory of Automorphic Functions,} Publ. Math. Soc. of Japan, 11, Iwanami Shoten and Princeton University Press, 1971,


    
\bibitem[Sho]{shioda} T. Shioda, {\em On elliptic modular surfaces},
  J. Math. Soc. Jap. 24  (1971)


  
\bibitem[SZ]{sz} J. Steenbrink, S. Zucker, {\em Variation of mixed Hodge structure. I.} Invent. Math. 80 (1985), no. 3, 489–542.
  

\bibitem[T]{tate} J. Tate,  {\em Conjectures on algebraic cycles in l-adic cohomology}, pp. 71–83 in Motives (Seattle, 1991), edited by U.
Jannsen et al., Proc. Sympos. Pure Math. 55, AMS (1994) 

\bibitem[VZ]{vz} E. Viehweg, K. Zuo, {\em A characterization of
    certain Shimura curves in the moduli stack of abelian varieties },
  J. Diff. Geom 66 (2004)
  
\bibitem[Z]{zucker} S. Zucker,  {\em Hodge theory  with degenerating
    coefficients}, Annals Math 109  (1979)
\end{thebibliography}
\end{document}